\documentclass[preprint,3p]{elsarticle}
\usepackage{amsmath, amsthm, amscd, amsfonts, amssymb, graphicx, color}

\newtheorem{example}{Example}[section]

\newtheorem{proposition}{Proposition}[section]
\newtheorem{theorem}{Theorem}[section]
\newtheorem{lemma}{Lemma}[section]

\numberwithin{equation}{section}
\newproof{pf}{Proof}
\newproof{pot}{Proof of Lemma \ref{l2}}
\newproof{poot}{Proof of Corollary \ref{co1}}
\newproof{pft5}{Proof of Theorem \ref{t5}}
\newproof{pft1}{Proof of Theorem \ref{t1}}
\newproof{pft6}{Proof of Theorem \ref{t6}}
\newproof{pft3}{Proof of Theorem \ref{t3}}

\numberwithin{equation}{section}
\newdefinition{ex}{Example}[section]
\journal{}
\begin{document}
\begin{frontmatter}

\title{Bounds for the extremal parameter of nonlinear eigenvalue problems and application to the explosion problem in a flow}
\author{A. Aghajani $^{\text{a,b}}$ \corref{cor1}}
\cortext[cor1]{Corresponding author: A. Aghajani. Tel.
+9821-73913426. Fax +9821-77240472.} \ead{aghajani@iust.ac.ir}

\author{A.M. Tehrani $^{\text{a}}$\corref{}}
\ead{amtehrani@iust.ac.ir}

\address {$^{\text{a}}$School of Mathematics, Iran University of Science and
Technology, Narmak, Tehran 16844-13114, Iran.\\

$^{\text{b}}$School of Mathematics, Institute for Research in Fundamental Sciences (IPM), P.O.Box: 19395-5746, Tehran, Iran. }
\begin{abstract}
We consider the nonlinear eigenvalue problem $ L u = \lambda f(u) $, posed in a smooth bounded domain $ \Omega \subseteq \Bbb{R}^{N} $ with Dirichlet boundary condition, where $ L $ is a uniformly elliptic second-order linear differential operator, $ \lambda > 0 $ and $ f:[0,a_{f}) \rightarrow \Bbb{R}_{+} $ $ (0 < a_{f} \leqslant \infty)$ is a smooth, increasing and convex nonlinearity such that $ f(0) > 0 $ and which blows up at $ a_{f} $. First we present some upper and lower bounds for the extremal parameter $ \lambda^{*} $ and the extremal solution $ u^{*} $. Then we apply the results to  the operator $ L_A = - \Delta + A c(x) $ with $ A>0 $  and $ c(x) $ is a divergence-free flow in $ \Omega $. We show that, if $\psi_{A,\Omega}$ is the maximum of the solution $\psi_{A}(x)$ of the equation  $ L_A u = 1 $ in $\Omega$ with Dirichlet boundary condition, then  for  any incompressible flow $ c(x) $ we have, $\psi_{A,\Omega} \longrightarrow 0$ as $A \longrightarrow \infty$ if and only if $c(x)$ has no non-zero first integrals in $H_{0}^{1}(\Omega)$. Also, taking $ c(x)=-x\rho(|x|) $ where $\rho$ is a smooth real function on $[0,1]$ then $c(x)$  is never divergence-free in unit ball $ B\subset \Bbb{R}^{N} $, but our results completely determine the behaviour of the extremal parameter $ \lambda^{*}_{A} $ as $ A \longrightarrow \infty $.
\end{abstract}

\begin{keyword}
semilinear elliptic problem, nonlinear eigenvalue problem, extremal solution.

\textit{2010 Mathematics Subject Classification:}
35B40, 35P30, 35J91, 35B32
\end{keyword}

\end{frontmatter}
%
\section{Introduction and main results}
%
The explosion problem in a flow concerns existence and regularity of positive solutions of nonlinear eigenvalue problem of the form
\begin{equation}\label{eq14}
\left\{\begin{array}{ll} - \Delta u + c(x) \centerdot \nabla u = \lambda f(u)  {\rm }\ &x \in \Omega, \\
\hskip23.6mm u = 0 {\rm }\ &x \in \partial \Omega,
\end{array}\right.
\end{equation}
where $\Omega$ is a bounded smooth domain in $R^{N}$ ($N \geqslant 2$), $\lambda > 0$, $f : [0,a_{f}) \rightarrow \Bbb{R}_{+}$ is a smooth, increasing, convex function such that $f(0) > 0$, $\int_{0}^{a_{f}}\frac{ds}{f(s)} < \infty$ which blows up at the endpoint of its domain. We consider two cases either $f$ is a \emph{regular nonlinearity} i.e., $D_{f}:=[0,+\infty)$ and $f$ is superlinear, namely $f(t) / t \rightarrow \infty$ as $t \rightarrow \infty$, or when $D_{f}:=[0,1)$ and $\lim_{t\nearrow 1}f(t)= + \infty$
called a \emph{singular nonlinearity}. Typical examples of regular nonlinearities $f$ are $e^{u}$, $(1+u)^{p}$ for $p > 1$, while singular nonlinearities include $(1-u)^{-p}$ for $p > 1$.

It is said that a solution of problem $(\ref{eq14})$ is classical provided $\|u\|_{L^{\infty}} < \infty$ (resp., $\|u\|_{L^{\infty}} < 1$) if $f$ is a regular (resp., singular) nonlinearity. It is known that there exists an extremal parameter (critical threshold) $\lambda^{*} \in (0, \infty)$ depending on $\Omega$, $c(x)$ and $N$, such that problem $(\ref{eq14})$ has a unique minimal classical solution $u_{\lambda} \in C^{2}(\overline{\Omega})$ if $0 < \lambda < \lambda^{*}$ while no solution exists, even in the weak sense, for $\lambda > \lambda^{*}$. One can show that $ \lambda \longmapsto u_{\lambda}(x) $ is increasing in $ \lambda $ for all $ x \in \Omega $ and therefore one can define the extremal solution $ u^{*}(x) = \lim_{\lambda \nearrow \lambda^{*}} u_{\lambda}(x) $, which is a weak solution of problem $ (\ref{eq14}) $ at $ \lambda = \lambda^{*} $. The regularity of solutions at $ \lambda=\lambda^{*} $ is a delicate issue. In the case that endpoint of the domain $ f $ is finite, Cowan and Ghoussoub in \cite{13} proved that the extremal solution of problem $ (\ref{eq14}) $ with $ f(u) = \frac{1}{(1-u)^{2}} $ is regular for all $1 \leqslant N \leqslant 7$. Luo, Ye and Zhou in \cite{11} proved that the extremal solution is regular in the low-dimensional case. In particular, for the radial case, all extremal solutions are regular in dimension two. When $ c \equiv 0 $, the regularity of $u^*$ has been studied extensively in the literature \cite{17, 2, 4, 13, 18, 3, 16}. For example, we know that when $ f(u) = e^{u} $ or $ f(u) = (1+u)^{p} $, then $u^*$ is regular in dimensions $ N\leqslant 9 $.
 For general nonliearities $f$,  Nedev \cite{3} proved the regularity of $u^*$ in dimensions $ N = 2, 3 $.  In dimension $N=4$ the same is proved by Cabr\'e \cite{5} when $\Omega$ is convex (without assuming the convexity of $f$), and by Villegas \cite{19} for arbitrary domains and $f$ is convex. However, it is still an open problem to establish the regularity of $u^*$  in dimensions $5\leqslant N \leqslant 9$ for regular nonlinearities $f$. Ghoussoub and Guo in \cite{14} showed that when $ \Omega $ is a ball and $ f(u) = \frac{1}{(1-u)^{2}} $, then $u^*$ is singular if $ N \geqslant 8 $, while it is regular if $ N < 8 $. 

In this work, first we consider semilinear second-order elliptic equation of the form
\begin{equation}\label{eq1}
\left\{\begin{array}{ll} L u = \lambda f(u)  {\rm }\ &x \in \Omega, \\
\hskip2.5mm u = 0 {\rm }\ &x \in \partial \Omega,
\end{array}\right.
\end{equation}
where $L$ is a second-order linear differential operator acting on functions $u : \Omega \rightarrow \Bbb{R}$ which is uniformly elliptic and has the following nondivergence general form
\begin{eqnarray*}
Lu = - \sum_{i,j=1}^{N} a_{ij}(x)\frac{\partial^{2}u}{\partial x_{i} \partial x_{j}} + c(x) \centerdot \nabla u,
\end{eqnarray*}
where $c(x) = \big( c_{1}(x),c_{2}(x), ... , c_{n}(x) \big)$ is a smooth vector field on $\overline{\Omega}$ and $a_{i,j}(x)=a_{j,i}(x)$ are smooth functions. The linear operator $L$ can be also showed in the divergence form as
\begin{eqnarray*}
L u = - \sum_{i,j=1}^{N} \dfrac{\partial}{\partial x_{j}} \bigg( a_{i,j}(x) \dfrac{\partial u}{\partial x_{i}} \bigg) + b(x) \centerdot \nabla u,
\end{eqnarray*}
where $b(x) = \big( b_{1}(x),b_{2}(x),...,b_{N}(x) \big)$ and $b_{i}(x) = c_{i}(x) + \displaystyle{\sum_{j=1}^{N}}\dfrac{\partial a_{i,j}(x)}{\partial x_{j}}$ for all $1 \leqslant i \leqslant N$. When the linear operator $L$ has divergence form the linear operator $L^{*}$, the formal adjoint of $L$, is
\begin{eqnarray*}
L^{*} u = - \sum_{i,j=1}^{N} \dfrac{\partial}{\partial x_{i}} \bigg( a_{i,j}(x) \dfrac{\partial u}{\partial x_{j}} \bigg) - \textrm{div} \big( u(x)  b(x)  \big).
\end{eqnarray*}
Fredholm alternative theorem and regularity theory imply that the following equation
\begin{equation*}
\left\{\begin{array}{ll} L u = 1  {\rm }\ &x \in \Omega, \\
\hskip2.5mm u = 0 {\rm }\ &x \in \partial \Omega,
\end{array}\right.
\end{equation*}
has a unique nonnegative smooth solution \cite{15}. This solution will be denoted by $\psi_{L}$ and will be called the torsion function for uniformly elliptic operator $L$. If $L=-\Delta$, then we omit $L$ and just write $\psi$. We shall denote $\psi_{L,\Omega}:=\sup_{x \in \Omega}\psi_{L}(x)=\|\psi_{L}\|_{\infty}$ and $\psi_{\Omega}:=\sup_{x \in \Omega}\psi (x)=\|\psi\|_{\infty}$. We also denote by $(\eta(x),\mu_{1}(L^{*},\Omega))$, the first eigenpair of adjoint problem
\begin{equation} \label{eq25}
\left\{\begin{array}{ll} L^{*}\eta  = \mu_{1}(L^{*},\Omega) \eta  {\rm }\ &x \in \Omega, \\
\hskip4mm \eta = 0 {\rm }\ &x \in \partial \Omega,
\end{array}\right.
\end{equation}
A nonnegetive solution $u_{\lambda}(x)$ of $(\ref{eq1})$ is said to be minimal if for any other solution $v$ of $(\ref{eq1})$ we have $u_{\lambda}(x) \leqslant v(x)$ for all $x \in \Omega$. Also, we say that a solution $v(x)$ of $(\ref{eq1})$ is stable if the principal eigenvalue $\kappa_{1}$ of the linearized operator $\tilde{L}_{\lambda} \varphi = L \varphi - \lambda f'(v) \varphi$ is positive.

Fix a flow profile $c(x)$ and consider the following problem
\begin{equation}\label{eq12}
\left\{\begin{array}{ll} - \Delta u + A c(x) \centerdot \nabla u = \lambda f(u)  {\rm }\ &x \in \Omega, \\
\hskip26.3mm u = 0 {\rm }\ &x \in \partial \Omega,
\end{array}\right.
\end{equation}
where $A$ is a positive number. Denote by $\lambda^{*}(A)$, $\psi_{A}$ and $\psi_{A,\Omega}$, the extremal parameter of problem $(\ref{eq12})$, the torsion function for the linear operator $L_{A}=-\Delta + A c(x).\nabla$ and $\psi_{A,\Omega} = \sup_{x \in \Omega} \psi_{A}(x)$, respectively.

H. Berestycki and collaborators \cite{1}, by using the ideas from \cite{8,9,10}, showed that in problem $(\ref{eq12})$ when $ c(x) $ is divergence-free (incompressible) i.e., $ \textrm{div} ~ c(x) =0 $, then
\\\\
\textbf{Theorem A.}
We have $\lambda^{*}(A) \longrightarrow \infty$ as $A \longrightarrow \infty$ if and only if $u$ has no non-zero first integrals in $H_{0}^{1}(\Omega)$.

Recall that a function $ \Psi \in H^{1}(\Omega) $ is a first integral of $ u $ if $ u \centerdot \nabla \Psi = 0 $ a.e. in $ \Omega $. They also proved that $\psi_{A,\Omega} \longrightarrow 0$ as $A \longrightarrow \infty$ if $c(x)$ has no first integrals in $H_{0}^{1}(\Omega)$ (see Lemma 3.2 in \cite{1} ). Indeed, the proof of their result based on the key observation that one can write $ \psi_{A}(x) = \int_{0}^{\infty} \xi (t,x) dt $ where the function $ \xi (t,x) $  solves a special parabolic problem on $ [0,\infty) \times \Omega $ discussed in \cite{12}. In this paper, we prove the condition that $\psi_{A,\Omega} \longrightarrow 0$ as $A \longrightarrow \infty$ is also sufficient (see the following theorem) and we give a rather simple proof for the necessary condition using only the maximum principle.
\begin{theorem}\label{t3}
For any incompressible flow $ c(x) $ in problem $ (\ref{eq12}) $ we have $\psi_{A,\Omega} \longrightarrow 0$ as $A \longrightarrow \infty$, if and only if $c(x)$ has no non-zero first integrals in $H_{0}^{1}(\Omega)$.
\end{theorem}
Another illustration of how our results are applicable, we consider semilinear second-order elliptic equations of the form
\begin{equation}\label{eq33}
\left\{\begin{array}{ll} - \triangle u - A \rho (|x|) x \centerdot \nabla u = \lambda f(u)  {\rm }\ &x \in B, \\
\hskip31mm u = 0 {\rm }\ &x \in \partial B,
\end{array}\right.
\end{equation}
where $ B:=B(0,1) $, $A \geqslant 0$, $\lambda > 0$, $\rho : [0,1] \rightarrow \Bbb{R}$ is a smooth function and $c(x):=-x \rho (|x|)$, $ x \in B $ is a smooth vector field and $f:[0,a_{f}) \rightarrow \Bbb{R}_{+}$ is regular or singular nonlinearity. Notice that $ c(x) $ is never divergence-free as $\textrm{div} ~ c(x) =0$ implies that $\rho (|x|) = \dfrac{a}{|x|^{N}}$ $(x \neq 0)$ for some constant $a$ which is impossible, because $\rho$ is assumed to be continuous on $[0,1]$.

The following theorem, completely determine the behavior of extremal parameter of problem $(\ref{eq33})$.
\begin{theorem}\label{t6}
Consider problem $(\ref{eq33})$, then
\begin{itemize}
\item[$(i)$] If there exits $x_{0} \in [0,1]$ such that $\rho(x_{0}) < 0$, then $\psi_{L_{A},B} \longrightarrow \infty$ as $A \longrightarrow \infty$. This implies that for all nonlinearities $f$  we have $\lambda^{*}(A) \longrightarrow 0$ as $A \longrightarrow \infty$.
\item[$(ii)$]If $\rho \geqslant 0$ and $\rho \not\equiv 0$ on any interval $I \subseteq [0,1]$, then $\psi_{L_{A},B} \longrightarrow 0$ as $A \longrightarrow \infty$. This implies that for all nonlinearities $f$  we have $\lambda^{*}(A) \longrightarrow \infty$ as $A \longrightarrow \infty$.
\item[$(iii)$] If $\rho \geqslant 0$ and $\rho \equiv 0$ on some interval $I \subseteq [0,1]$, then there exits positive constant $C_{N,\rho} $ where $ C_{N,\rho} $ depends on $\rho,N$ and independent of $A$ such that
\begin{eqnarray*}
C_{N,\rho} \leqslant \psi_{L_{A},B} \leqslant \dfrac{1}{2 N} ~~~ \textrm{for} ~ \textrm{all} ~ A \geqslant 0.
\end{eqnarray*}
Consequently, for all nonlinearities $f$  there exist  positive constants $D_{N,f},\tilde{D}_{N,\rho,f}$ where $ D_{N,f} $ depends on $ N $ and  $f$ but not $A$ and $ \tilde{D}_{N,\rho,f} $ depends on $ \rho, N $ and $ f $ but not   $ A $ such that
    $$D_{N,f} \leqslant \lambda^{*}(A) \leqslant \tilde{D}_{N,\rho,f}~~\text{ for~ all }A \geqslant 0.$$
\end{itemize}
\end{theorem}
The authors in \cite{1} also proved that the critical threshold $ \lambda^{*} $ for $ (\ref{eq14}) $ when $ c(x) $ is incompressible cannot close to zero, precisely, for any domain $ \Omega $ and regular nonlinearity $ f $ there exists $ \lambda_{0} > 0 $ so that the extremal parameter $ \lambda^{*} $ of problem $ (\ref{eq14}) $ satisfies $ \lambda^{*} \geqslant \lambda_{0} > 0 $ for all incompressible flows $ c(x) $ in $ \Omega $. The constant $ \lambda_{0} $ depends on $ \Omega $ and the function $ f $. They also showed that this result does not hold without the restriction that the flow $ c(x) $ is incompressible and give an example $ c_{n}(x) = 4 n x$ for all $ n \in \Bbb{N} $ such that $ c_{n}(x) $ is never divergence-free and the critical threshold for $ (\ref{eq14}) $ tends to zero as $ n $ tends to infinity. To show this in \cite{1} (in dimension two and $ \Omega = B $), by setting $ \Psi_{n} = e^{-n |x|^{2}} \Theta_{n} $ where $ \Theta_{n} $ is a radial solution of problem $ (\ref{eq12}) $ with $ c_{n}(x) = 4 n x$ for some $ \lambda_{n} $, they obtained a self-adjoint problem for $ \Theta_{n} $, then using suitable test function in the variational principle for $ \lambda_{n} $ proved that $ \lambda_{n} \leqslant C e^{-c n} \longrightarrow 0 $ as $ n \longrightarrow\infty $ which implies that $ \lambda^{*}_{n} \longrightarrow 0 $ as well. This result, however, is a direct consequence of our Theorem $ \ref{t6} $ part $ (i) $ by taking $ \rho (|x|) = -4. $

In this paper, before proving Theorems $ \ref{t3} $ and $ \ref{t6} $, we consider the general semilinear eigenvalue problem $ (\ref{eq1}) $ and shall present some sharp upper and lower bounds for the extremal parameter for the general nonlinearity $ f $ (regular or singular) as well as pointwise lower and upper bounds on the minimal stable solution $ u_{\lambda} $ of $ (\ref{eq1}) $.
Our first proposition establishes the existence as well as lower and upper bounds of the extremal parameter of problem $(\ref{eq1})$.
\begin{proposition}\label{p1}
There exists $\lambda^{*}(L,\Omega,f) \in (0,\infty)$ such that:
\begin{itemize}
\item[$(i)$] for every $0 < \lambda < \lambda^{*}(L,\Omega,f)$ the problem $(\ref{eq1})$ has a unique positive classical solution $u_{\lambda}(x)$ which is minimal and stable. Furthermore, this extremal parameter satisfies
\begin{eqnarray}\label{eq8}
\dfrac{1}{\psi_{L,\Omega}} \sup_{0 < t < a_{f}} \dfrac{t}{f(t)} \leqslant \lambda^{*}(L,f,\Omega) \leqslant \mu_{1}(L^{*},\Omega)  \sup_{0 < t < a_{f}} \dfrac{t}{f(t)}.
\end{eqnarray}
\item[$(ii)$] for each $x \in \Omega$, the function $\lambda \longmapsto u_{\lambda}(x)$ is differentiable and strictly increasing on $(0,\lambda^{*})$.
\item[$(iii)$] there exits no classical solution of $(\ref{eq1})$ for $\lambda > \lambda^{*}(L,\Omega,f)$.
\end{itemize}
\end{proposition}
The proof of this result is very close to that in \cite{1}, but for the convenience of the reader we present it in this paper. In the following theorem, we give another upper bound for the extremal parameter of problem $(\ref{eq1})$ which, in many cases, represent a sharper upper bound than $(\ref{eq8})$. We also give pointwise lower bound for the extremal solution of problem $(\ref{eq1})$. Throughout this paper, for all nonlinearity $ f : [0,a_{f} ) \rightarrow \Bbb{R}_{+} $, we define the function $ F:[0,a_{f}] \rightarrow \Bbb{R}_{+} $ as follows
\begin{eqnarray}\label{eq37}
F(t)=\int_{0}^{t}\dfrac{\textrm{d} s}{f(s)}.
\end{eqnarray}
\begin{theorem}\label{t4}
Let $u \in C^{2}(\Omega)$ be a solution of problem $(\ref{eq1})$, then
\begin{eqnarray*}
F^{-1} \big( \lambda \psi_{L}(x) \big) \leqslant u(x) ~~~ \textrm{for} ~ \textrm{all} ~ x \in \Omega,
\end{eqnarray*}
where $ F $ is defined in $ (\ref{eq37}) $. Therefore if $x_{0} \in \Omega$ such that $\psi_{L}(x_{0}) = \psi_{L,\Omega}$, then
\begin{eqnarray*}
\lambda \leqslant \dfrac{F \big( u(x_{0}) \big)}{\psi_{L,\Omega}}.
\end{eqnarray*}
In particular, we have
\begin{eqnarray}\label{eq32}
\lambda^{*} \leqslant \dfrac{F \big( u^{*}(x_{0}) \big)}{\psi_{L,\Omega}} \leqslant \dfrac{F \big( a_{f} \big)}{\psi_{L,\Omega}} ~~~ \textrm{and} ~~~ F^{-1} \big( \lambda^{*} \psi_{L}(x) \big) \leqslant u^{*}(x) ~~~ \textrm{for} ~ \textrm{all} ~ x \in \Omega.
\end{eqnarray}
\end{theorem}
To see the sharpness of above results, consider the following problem
\begin{equation}\label{eq31}
\left\{\begin{array}{ll} L u = \lambda f(u^{p})  {\rm }\ &x \in \Omega, \\
\hskip2.5mm u = 0 {\rm }\ &x \in \partial \Omega,
\end{array}\right.
\end{equation}
where $p \geqslant 1$ and $f:\Bbb{R}_{+} \rightarrow \Bbb{R}_{+}$ is an increasing, convex and superlinear $C^{2}$-function such that $f(0) > 0$. In the following theorem, we show that upper bound $(\ref{eq32})$ for the extremal parameter of problem $(\ref{eq31})$ is arbitrarily close to lower bound $(\ref{eq8})$ provided that $p$ is sufficiently large. This also implies that upper bound $(\ref{eq32})$ is an improvement of $(\ref{eq8})$.
\begin{theorem}\label{t5}
Consider semilinear second-order elliptic equation $(\ref{eq31})$. Then
\begin{eqnarray*}
\lim_{p \rightarrow \infty} \lambda^{*}_{p} = \frac{1}{f(0)\psi_{L,\Omega}} \hskip5mm \textrm{and} \hskip5mm \lim_{p \rightarrow\infty} \|u^{*}_{p}\|_{\infty} = +\infty,
\end{eqnarray*}
where $\lambda^{*}_{p}$ and $u^{*}_{p}$ are the extremal parameter and extremal solution $($respectively$)$ of problem $(\ref{eq31})$.
\end{theorem}
In the following theorem, we give another lower bound for the extremal parameter of problem $(\ref{eq1})$ which is a better lower bound, at least when $ L = - \Delta $, than $(\ref{eq8})$ for more values of $ N $. We also give pointwise upper bound for the minimal solution of problem $(\ref{eq1})$ for all $\lambda \in (0,\overline{\lambda})$ where $\overline{\lambda} \leqslant \lambda^{*}$ is given in below.
\begin{theorem}\label{t1}
Consider the semilinear elliptic equation $(\ref{eq1})$, then
\begin{equation}\label{eq10}
\lambda^{*}(L,f,\Omega) \geqslant \sup_{0 < \alpha < \frac{||F||_{\infty}}{\psi_{L,\Omega}}} \alpha-\alpha^{2}\beta(\alpha) := \overline{\lambda},
\end{equation}
where $\beta(\alpha):=\displaystyle{\sup_{x \in \Omega}}f' \big( F^{-1}(\alpha \psi_{L}(x)) \big) \big| \nabla \psi_{L}(x) \big|^{2} $ and $ F $ is defined in $ (\ref{eq37}) $. Furthermore, if we define $\lambda (\alpha) = \alpha - \alpha^{2} \beta(\alpha)$ for all $0 \leqslant \alpha \leqslant \|F\|_{\infty} / \psi_{L,\Omega}$, then
\begin{eqnarray*}
u_{\lambda(\alpha)}(x) \leqslant F^{-1} \big( \alpha \psi_{L}(x) \big) ~~~ \textrm{for} ~ \textrm{all} ~ 0 \leqslant \alpha \leqslant \dfrac{\|F\|_{\infty}}{\psi_{L,\Omega}}.
\end{eqnarray*}
\end{theorem}
The authors in \cite{6} show that lower bound $ (\ref{eq10}) $ gives the exact value of the extremal parameter $ \lambda ^{*} $ when $ L = - \Delta $, $ f(u) = e ^{u} $, $ f(u) = (1+u)^{p} $ and $ f(u) = (1-u)^{-p} $ in some dimensions.

Using the above theorems we get
\begin{proposition}\label{p2}
Assume that $u_{\lambda}$ is the minimal solution of problem $(\ref{eq1})$ and $ F $ is defined in $ (\ref{eq37}) $, then
\item[$(i)$] for each $x \in \Omega$, the function $\lambda \longmapsto\dfrac{F \big( u_{\lambda}(x) \big)}{\lambda}$ is increasing on $(0,\lambda^{*})$. In particular,
$$ u_{\lambda}(x) \leqslant F^{-1} \big( \dfrac{\lambda}{\lambda^{*}} \|F\|_{\infty} \big) ~ for ~ all ~ \lambda \in (0,\lambda^{*}).$$
\item[$(ii)$] $\dfrac{F \big( u_{\lambda}(x) \big)}{\lambda} \longrightarrow \psi_{L}$ uniformly as $\lambda \longrightarrow 0^{+}$.
\end{proposition}
Note that the first assertion of Proposition $ \ref{p2} $ gives an upper bound for the minimal solution of problem $ (\ref{eq1}) $ which is an interesting issue in itself. For example, consider the following problem
\begin{equation*}
\left\{\begin{array}{ll} - \Delta u = \lambda e^{u}  {\rm }\ &x \in \Omega, \\
\hskip5mm u = 0 {\rm }\ &x \in \partial \Omega.
\end{array}\right.
\end{equation*}
Here we have $ f(t)=e^{t} $, $ F(t)=1-e^{-t} $, $ \|F\|_{\infty}=1 $ and $ F^{-1}(t)=\ln \dfrac{1}{1-t} $. Taking $ \lambda^{*}=\lambda^{*}(e^{t},\Omega) $, then, by part (i) of Proposition $ \ref{p2} $ we have
\begin{eqnarray*}
u_{\lambda}(x) \leqslant \ln \dfrac{\lambda^{*}}{\lambda^{*} - \lambda} ~~~ \textrm{for} ~ \textrm{all} ~ \lambda \in (0,\lambda^{*}).
\end{eqnarray*}
If $ N > 9 $ and $ \Omega = B(0,1) $, then $ \lambda^{*}(e^{t},B(0,1)) = 2N-4 $ \cite{2}, so we have
\begin{eqnarray*}
u_{\lambda}(x) \leqslant \ln \dfrac{2N-4}{2N-4 - \lambda} ~~~ \textrm{for} ~ \textrm{all} ~ \lambda \in (0,2N-4).
\end{eqnarray*}
%
\section{Existence and basic properties of the extremal parameter}
%
In this section, we prove Proposition $\ref{p1}$ which is well known when $L=-\Delta$, and also prove the first assertion of Proposition $\ref{p2}$. To do these, first we give a nonexistence result for the nonlinear eigenvalue problem $(\ref{eq1})$.
\begin{lemma}\label{l4}
The problem $(\ref{eq1})$ admits no classical solutions for $\lambda > \mu_{1}(L^{*},\Omega) \displaystyle{\sup_{0 < t < a_{f}}} \dfrac{t}{f(t)}$.
\end{lemma}
\begin{proof}
Clearly,
\begin{eqnarray*}
\int_{\Omega} \Big( L^{*} \eta - \mu_{1}(L^{*},\Omega) \eta \Big) u ~ \textrm{d} x = 0,
\end{eqnarray*}
for any solution $ u $ of $ (\ref{eq1}) $. Now, integration by parts implies that
\begin{eqnarray*}
\int_{\Omega} \eta \Big( \lambda f(u) - u \mu_{1}(L^{*},\Omega)) dx =0,
\end{eqnarray*}
and thus there exists $x \in \Omega$ such that $\lambda f(u(x)) - u(x) \mu_{1}(L^{*},\Omega) < 0$. It follows that
\begin{eqnarray*}
\lambda \leqslant \mu_{1}(L^{*},\Omega) \sup_{0 < t < a_{f}} \dfrac{t}{f(t)}.
\end{eqnarray*}
This completes the proof.
\end{proof}
Now, we show that there exists a constant $ C > 0 $ such that for all $ \lambda \in ( 0 , C ) $ the problem $ ( \ref{eq1} ) $ has a positive classical solution.
\begin{lemma}\label{l2}
Problem $ ( \ref{eq1} ) $ admits a minimal nonnegative solution $u_{\lambda}(x)$ for all $\lambda \leqslant \dfrac{1}{\psi_{L,\Omega}} \displaystyle{\sup_{0 < t < a_{f}}}\dfrac{t}{f(t)}$.
\end{lemma}
To prove Lemma $\ref{l2}$, we construct a super-solution and using it we show that a positive solution of $(\ref{eq1})$ exists. To do that, we need the following well-known fact.
\begin{lemma}\label{l3}
Suppose that there exists a smooth function $\overline{u}(x)$ satisfying
\begin{equation}\label{eq3}
\left\{\begin{array}{ll} L \overline{u} \geqslant \lambda f(\overline{u})  {\rm }\ &x \in \Omega, \\
\hskip2.5mm \overline{u} \geqslant 0 {\rm }\ &x \in \partial \Omega.
\end{array}\right.
\end{equation}
Then there exists a classical solution $u_{\lambda}$ of $(\ref{eq1})$ which is minimal.
\end{lemma}
\begin{proof}
Let $u_{0} \equiv 0$ and define an approximating sequence $u_{n}(x)$ such that $u_{n+1}(x)$ is the smooth solution of
\begin{equation*}
\left\{\begin{array}{ll} L u_{n+1} = \lambda f(u_{n})  {\rm }\ &x \in \Omega, \\
\hskip2.5mm u_{n+1} = 0 {\rm }\ &x \in \partial \Omega.
\end{array}\right.
\end{equation*}
From the maximum principle we know that  $0\leq u_{0} \leq\overline{u} $. Now by induction, assuming $0 \leqslant u_{n-1} \leqslant \overline{u}$ for some $n \in \Bbb{N}$, we get
\begin{equation*}
\left\{\begin{array}{ll} L (\overline{u}-u_{n}) \geqslant \lambda [f(\overline{u})-f(u_{n-1})] \geqslant 0  {\rm }\ &x \in \Omega, \\
\hskip5mm \overline{u}-u_{n} \geqslant 0 {\rm }\ &x \in \partial \Omega,
\end{array}\right.
\end{equation*}
 concludes that $0 \leqslant u_{n} \leqslant \overline{u}$. In a similar way, the maximum principle implies that the sequence $\{ u_{n} \}$ is monotone increasing. Therefore, the sequence $\{ u_{n} \}$ converges uniformly to a limit $u_{\lambda}$ which has to be a classical solution of $(\ref{eq1})$ and satisfies $0 \leqslant u_{\lambda} \leqslant \overline{u}$. Since this inequality holds for any solution of $(\ref{eq3})$, then $u_{\lambda}$ is a minimal positive solution of $(\ref{eq1})$ and is clearly unique.
\end{proof}
\begin{proof}[Proof of Lemma \ref{l2}]
Choose $\alpha > 0$ such that
\begin{eqnarray*}
\dfrac{\alpha \psi_{L,\Omega}}{f(\alpha \psi_{L,\Omega})} = \sup_{0 < t < a_{f}} \dfrac{t}{f(t)},
\end{eqnarray*}
and consider the smooth function $\overline{u}(x)=\alpha \psi_{L}(x)$ for $x \in \Omega$. Clearly, we have
\begin{equation*}
\left\{\begin{array}{ll} L \overline{u} = \alpha \geqslant \lambda f(\overline{u})  {\rm }\ &x \in \Omega, \\
\hskip2.5mm \overline{u} = 0 {\rm }\ &x \in \partial \Omega,
\end{array}\right.
\end{equation*}
provided that $\lambda \leqslant \dfrac{\alpha}{f(\alpha \psi_{L,\Omega})} = \dfrac{1}{\psi_{L,\Omega}} \dfrac{\alpha \psi_{L,\Omega}}{f(\alpha \psi_{L,\Omega})} = \dfrac{1}{\psi_{L,\Omega}} \displaystyle{\sup_{0 < t < a_{f}}} \dfrac{t}{f(t)}$. Now, existence of a minimal solution to $(\ref{eq1})$ follows from Lemma $\ref{l3}$.
\end{proof}
The following two lemmas show that any minimal solution of $(\ref{eq1})$ is stable. We recall that for any minimal solution $u_{\lambda}$ of $(\ref{eq1})$ we denote by $\kappa_{1}(\lambda,u_{\lambda})$ the principal eigenvalue corresponding to positive eigenfunction $\phi$ of the following linearized operator $\tilde{L}_{\lambda}$
\begin{eqnarray}\label{eq7}
\tilde{L}_{\lambda} \varphi = L \varphi - \lambda f'(u_{\lambda}) \varphi ~~~ \textrm{for} ~ \textrm{all} ~ \varphi \in C^{2}(\Omega)
\end{eqnarray}
\begin{lemma}
For any minimal solution of $ ( \ref{eq1} ) $ we have $ \kappa_{1} ( \lambda , u_{\lambda} ) \geqslant 0 $.
\end{lemma}
\begin{proof}
Assume that $u_{\lambda}$ is a minimal solution of $(\ref{eq1})$ and the principal eigenvalue $\kappa_{1}(\lambda,u_{\lambda})$ of the problem
\begin{equation*}
\left\{\begin{array}{ll} L \phi - \lambda f'(u_{\lambda}) \phi = \kappa_{1}(\lambda,u_{\lambda}) \phi  {\rm }\ &x \in \Omega, \\
\hskip20.5mm \phi = 0 {\rm }\ &x \in \partial \Omega,
\end{array}\right.
\end{equation*}
is negative. Consider the function $\phi_{\epsilon}=u_{\lambda} - \epsilon \phi$, then we have
\begin{align*}
L \phi_{\epsilon} - \lambda f(\phi_{\epsilon}) &= \lambda f(u_{\lambda}) - \epsilon \lambda f'(u_{\lambda}) \phi - \epsilon \kappa_{1}(\lambda, u_{\lambda}) \phi - \lambda f(u_{\lambda} - \epsilon \phi) \\
&= - \epsilon \kappa_{1}(\lambda,u_{\lambda}) \phi + \lambda \big( f(u_{\lambda}) - \epsilon f'(u_{\lambda}) \phi - f(u_{\lambda} - \epsilon \phi ) \big) \\
&= - \epsilon \kappa_{1}(\lambda,u_{\lambda}) - \dfrac{\epsilon^{2} f''(\xi)}{2} \phi^{2} \geqslant 0,
\end{align*}
provided that $\epsilon$ is sufficiently small. This means that problem $(\ref{eq1})$ has a classical solution, say $u$, which satisfies $u \leqslant \phi_{\epsilon} < u_{\lambda}$ by Lemma $ (\ref{l3}) $. This contradicts the minimality of $u_{\lambda}$. So, we have $\kappa_{1}(\lambda,u_{\lambda}) \geqslant 0$ if $u_{\lambda}$ is a minimal solution.
\end{proof}
\begin{lemma}
Let $u_{\lambda}$ be a solution of $(\ref{eq1})$ such that $\kappa_{1}(\lambda,u_{\lambda}) = 0$. Then no classical solution of $(\ref{eq1})$ with $\overline{\lambda} > \lambda$ exists.
\end{lemma}
\begin{proof}
We argue by contradiction. Suppose that $\overline{\lambda} > \lambda$ and there exists a function $\overline{u} \geqslant 0$ such that
\begin{equation*}
\left\{\begin{array}{ll} L \overline{u} =\overline{\lambda} f(\overline{u})  {\rm }\ &x \in \Omega, \\
\hskip2.5mm \overline{u} = 0 {\rm }\ &x \in \partial \Omega.
\end{array}\right.
\end{equation*}
Also, denote by $\phi$ the positive eigenfunction of the adjoint problem
\begin{equation}\label{eq4}
\left\{\begin{array}{ll} L^{*} \phi = \lambda f'(u_{\lambda}) \phi  {\rm }\ &x \in \Omega, \\
\hskip4mm \phi = 0 {\rm }\ &x \in \partial \Omega.
\end{array}\right.
\end{equation}
Set $\eta_{\tau} = u_{\lambda} + \tau (\overline{u} - u_{\lambda})$ for all $\tau \in [0,1]$. Then convexity of $f$ implies that
\begin{align}\label{eq5}
L \eta_{\tau} - \lambda f(\eta_{\tau}) &= L \eta_{\tau} - \lambda f(\tau \overline{u} + (1-\tau)u_{\lambda} ) \\
&\geqslant L \eta_{\tau} - \lambda \tau f(\overline{u}) - \lambda (1-\tau) f(u_{\lambda}) \nonumber \\
&= \tau f(\overline{u}) (\overline{\lambda} - \lambda ) \geqslant 0, \nonumber
\end{align}
for all $\tau \in [0,1]$. Moreover, $L \eta_{0} = \lambda f(\eta_{0})$. If we differentiate $(\ref{eq5})$ with respect to $ \tau $ at $ \tau = 0 $, then we have the following inequality for $ \xi = \overline{u} - u_{\lambda} $:
\begin{eqnarray}\label{eq6}
L \xi - \lambda f'(u_{\lambda}) \xi \geqslant (\overline{\lambda} - \lambda ) f(\overline{u}) > 0.
\end{eqnarray}
Multiplying $(\ref{eq6})$ by the eigenfunction $\phi$ of $(\ref{eq4})$ and integrating by part, one obtains
\begin{eqnarray*}
0 < \int_{\Omega} \phi \big( L \xi - \lambda f'(u_{\lambda}) \xi \big) dx = \int_{\Omega} \xi \big( L^{*} \phi - \lambda f'(u_{\lambda}) \phi \big) dx = 0,
\end{eqnarray*}
which is a contradiction. Therefore, there exists no classical solution of $(\ref{eq1})$ for $\overline{\lambda} > \lambda$ if $\kappa_{1}(\lambda , u_{\lambda}) = 0$.
\end{proof}
Notice that the above lemma also proves that the extremal parameter of problem $(\ref{eq1})$ can be determined by
\begin{eqnarray*}
\lambda^{*}(L,\Omega,f) = \sup \Big\{ \lambda > 0 : \textrm{the} ~ \textrm{minimal} ~ \textrm{solution} ~ u_{\lambda} ~ \textrm{of} ~ \textrm{problem} ~ (\ref{eq1}) ~ \textrm{is} ~ \textrm{stable} \Big\}.
\end{eqnarray*}
The following lemma completes the proof of Proposition $\ref{p1}$.
\begin{lemma}
Let $u_{\lambda}$ be the minimal solution of $(\ref{eq1})$ for $\lambda \in (0,\lambda^{*})$, then for each $x \in \Omega$ the function $\lambda \longmapsto u_{\lambda}(x)$ is strictly increasing and differentiable on $(0,\lambda^{*})$.
\end{lemma}
\begin{proof}
Suppose that $0 < \lambda_{1} < \lambda_{2} < \lambda^{*}$, then clearly we have $L u_{\lambda_{1}} = \lambda_{1} f(u_{\lambda_{1}}) \leqslant \lambda_{1} f(u_{\lambda_{2}}) = \dfrac{\lambda_{1}}{\lambda_{2}} \lambda_{2} f(u_{\lambda_{2}}) = \dfrac{\lambda_{1}}{\lambda_{2}} L u_{\lambda_{2}}$. This means that
\begin{equation*}
\left\{\begin{array}{ll} L ( u_{\lambda_{1}} - \dfrac{\lambda_{1}}{\lambda_{2}} u_{\lambda_{2}} ) \leqslant 0  {\rm }\ &x \in \Omega, \\
\hskip5mm u_{\lambda_{1}} - \dfrac{\lambda_{1}}{\lambda_{2}} u_{\lambda_{2}} = 0 {\rm }\ &x \in \partial \Omega.
\end{array}\right.
\end{equation*}
Now, maximum principle implies that $\dfrac{u_{\lambda_{1}}}{\lambda_{1}} \leqslant \dfrac{u_{\lambda_{2}}}{\lambda_{2}}$. It follows that $u_{\lambda_{1}} < u_{\lambda_{2}}$.

Fix $\lambda_{0} \in (0,\lambda^{*})$ and define the operator $P$ such that $P(\lambda,\Phi) = L \Phi -\lambda f(\Phi)$ for all $\lambda \in (0,\lambda^{*})$ and $\Phi \in C^{2}(\Omega) \cap C(\partial \Omega) $ such that $\Phi = 0$ on $\partial\Omega$.
Clearly, $P$ is a $C^{1}$ map and $P(\lambda_{0},u_{\lambda_{0}}) = 0$. On the other hand $d_{\Phi} P(\lambda_{0},u_{\lambda_{0}}) = \tilde{L}_{\lambda_{0}}$, where $\tilde{L}_{\lambda_{0}}$ is defined by $(\ref{eq7})$ and $d_{\Phi} P(\lambda_{0},u_{\lambda_{0}}) $ is derivative of the function $ p(\lambda , \Phi) $ with respect to $ \Phi $. Since $u_{\lambda_{0}}$ is stable, the linearized operator $\tilde{L}_{\lambda_{0}}$ is invertible. By the Implicit Function Theorem, $\lambda \longmapsto u_{\lambda}(x)$ is differentiable at $\lambda_{0}$ and by monotonicity, $\dfrac{d u_{\lambda}}{d \lambda}(x) \geqslant 0$ for all $x \in \Omega$. 
\end{proof}
In the following, we prove the first assertion of Proposition $\ref{p2}$.
\begin{proof}[Proof of Proposition $ \ref{p2} $. $(i)$]
Let $0 < \lambda_{1} < \lambda_{2} < \lambda^{*}$ be arbitrary and set $\alpha = \dfrac{\lambda_{1}}{\lambda_{2}}$. Consider the function $\overline{u}(x)=F^{-1} \Big( \alpha F \big(u_{\lambda_{2}}(x) \big) \Big)$ for all $x \in \Omega$. Note that since $\alpha < 1$ and the function $f'$ is increasing, then $f'(u_{\lambda_{2}}) - \alpha f'(\overline{u}) > 0$. Letting  $A:=[a_{i,j}(x)]_{i,j}$ which is a symmetric matrix and positive definite for all $x \in \Omega$, then it can be easily checked that
\begin{align*}
L \overline{u} &= - \sum_{i,j=1}^{N} a_{i,j}\dfrac{\partial^{2} \overline{u}}{\partial x_{i} \partial x_{j}} + c(x).\nabla \overline{u} \\
&=- \sum_{i,j=1}^{N} a_{i,j} \Big( \dfrac{\alpha^{2} f'(\overline{u})-\alpha f'(u_{\lambda_{2}})}{f^{2}(u_{\lambda_{2}})} \dfrac{\partial u_{\lambda_{2}}}{\partial x_{i}} \dfrac{\partial u_{\lambda_{2}}}{\partial x_{j}} + \dfrac{\alpha}{f(u_{\lambda_{2}})} \dfrac{\partial^{2} u_{\lambda_{2}}}{\partial x_{i} \partial x_{j}} \Big) f(\overline{u}) \\
&\quad + \dfrac{\alpha f(\overline{u})}{f(u_{\lambda_{2}})} c(x). \nabla u_{\lambda_{2}} \\
&=- \sum_{i,j=1}^{N} a_{i,j} \Big( \dfrac{\alpha^{2} f'(\overline{u})-\alpha f'(u_{\lambda_{2}})}{f^{2}(u_{\lambda_{2}})} \dfrac{\partial u_{\lambda_{2}}}{\partial x_{i}} \dfrac{\partial u_{\lambda_{2}}}{\partial x_{j}} \Big) f(\overline{u}) + \dfrac{\alpha f(\overline{u})}{f(u_{\lambda_{2}})} L u_{\lambda_{2}} \\
&= \Big( \alpha \nabla u_{\lambda_{2}} A (\nabla u_{\lambda_{2}})^{T} \dfrac{ f'(u_{\lambda_{2}}) - \alpha f'(\overline{u}) }{f^{2}(u_{\lambda_{2}})} + \lambda_{1} \Big) f(\overline{u}) \geqslant \lambda_{1} f(\overline{u}).
\end{align*}
It then follows that $\overline{u}(x)$ is a super-solution of
\begin{equation*}
\left\{\begin{array}{ll} L u = \lambda_{1} f(u)  {\rm }\ &x \in \Omega, \\
\hskip2.5mm u = 0 {\rm }\ &x \in \partial \Omega,
\end{array}\right.
\end{equation*}
Hence, by Lemma $\ref{l3}$, we have $u_{\lambda_{1}} \leqslant \overline{u}$, so $\dfrac{F(u_{\lambda_{1}})}{\lambda_{1}} \leqslant \dfrac{F(u_{\lambda_{2}})}{\lambda_{2}}$.
\end{proof}
Uniform $L^{\infty}$-bounds for the functions $u_{\lambda}$ at $\lambda=\lambda^{*}$ are difficult to obtain. In the following, we prove that when we are away from $\lambda^{*}$ a uniform $L^{\infty}$-bound exists which is not depend on the domain $ \Omega $ and the linear operator $L$.
\begin{theorem}
For any $ 0 < \delta < 1 $ we have
\begin{eqnarray*}
0 \leqslant u_{\lambda}(x) \leqslant C(\delta,f):=F^{-1}\Big( (1-\delta)\|F\|_{\infty} \Big)~ for ~all ~0 < \lambda \leqslant (1-\delta) \lambda^{*}.
\end{eqnarray*}
Note that $C(\delta,f)$  depends only on $\delta$ and nonlinearity $f(t)$ but not on the domain $\Omega$ or the linear operator $L$.
\end{theorem}
\begin{proof}
Fix $0 < \delta < 1$. Now, by Proposition $\ref{p2}$ $ (i) $, we have
\begin{eqnarray*}
0 \leqslant u_{\lambda} (x) \leqslant u_{(1-\delta) \lambda^{*}}(x) \leqslant F^{-1} \Big( \dfrac{(1-\delta)\lambda^{*}}{\lambda^{*}}\|F\|_{\infty} \Big) = F^{-1} \Big( (1-\delta) \|F\|_{\infty} \Big) = C(\delta ,f).
\end{eqnarray*}
as claimed.
\end{proof}
%
\section{Upper and lower bound for the extremal parameter}
%
In this section, we give another upper and lower bound for the extremal parameter of problem $(\ref{eq1})$ which are, in many cases, sharper  than those in $(\ref{eq8})$. In fact, we prove Theorems $\ref{t4}$, $\ref{t5}$ and $\ref{t1}$. We also give an estimate on $L^{\infty}$-bound for the extremal solution of problem $(\ref{eq1})$.
\begin{proof}[Proof of Theorem $ \ref{t4} $]
As before let $A = [a_{i,j}]_{i,j}$ which is positive definite symmetric matrix. By a simple computation we have
\begin{align*}
L F(u) =& F''(u) \Big( - \sum_{i,j=1}^{N} a_{i,j}(x) \frac{\partial u}{\partial x_{i}}  \frac{\partial u}{\partial x_{j}} \Big) + F'(u) \Big( - \sum_{i,j=1}^{N} a_{i,j}(x) \frac{\partial^{2} u}{\partial x_{i} \partial x_{j}} + c(x).\nabla u \Big) \\
=& \frac{f'(u)}{f^{2}(u)} ~ \nabla u ~ A ~ (\nabla u)^{T} + \frac{L u}{f(u)} = \lambda = L \big( \lambda \psi_{L} \big).
\end{align*}
It follows that $ L \Big( F \big( u(x) \big) - \lambda \psi_{L}(x) \Big) \geqslant 0 $ for all $x \in \Omega$. On the other hand, $F \big( u(x) \big) - \lambda \psi_{L}(x) \geqslant 0 $ on $ \partial \Omega $, hence, by the maximum principle we must have $ F \big( u(x) \big) \geqslant \lambda \psi_{L}(x) $ for all $ x \in \Omega $, so
\begin{eqnarray*}
F^{-1} \big( \lambda \psi_{L}(x) \big) \leqslant u(x) ~~~ \textrm{for} ~ \textrm{all} ~ x \in \Omega.
\end{eqnarray*}
Thus
\begin{eqnarray*}
\lambda \leqslant \dfrac{F \big( u(x_{0}) \big)}{\psi_{L,\Omega}}.
\end{eqnarray*}
In particular, the extremal solution of problem $(\ref{eq1})$ satisfies
\begin{eqnarray*}
F^{-1} \big( \lambda^{*} \psi_{L}(x) \big) \leqslant u^{*}(x) ~~~ \textrm{for} ~ \textrm{all} ~ x \in \Omega.
\end{eqnarray*}
Hence
\begin{eqnarray*}
\lambda^{*} \leqslant \dfrac{F \big( u^{*}(x_{0}) \big)}{\psi_{L,\Omega}} \leqslant \dfrac{F \big( a_{f} \big)}{\psi_{L,\Omega}}.
\end{eqnarray*}
This completes the proof.
\end{proof}
Now, we give an estimate on $L^{\infty}$-bound for the extremal solution of problem $(\ref{eq1})$.
\begin{theorem}
Extremal solution of problem $(\ref{eq1})$ satisfies the following
\begin{eqnarray}
\|f'(u^{*})\|_{\infty} \geqslant \inf_{0 < t < a_{f}} \dfrac{f(t)}{t}.
\end{eqnarray}
\end{theorem}
\begin{proof}
If $u^{*}$ is singular, then the result is trivial. So we assume $u^{*}$ is regular. Let $\eta(x)$ be the positive first eigenfunction with corresponding eigenvalue $\mu_{1}(L^{*},\Omega)$ (see problem $(\ref{eq25})$). Now, since $u^{*}$ is regular there is some $\phi > 0$ such that
\begin{equation*}
\left\{\begin{array}{ll} L \phi = \lambda^{*} f'(u^{*}) \phi  {\rm }\ &x \in \Omega, \\
\hskip2.5mm \phi = 0 {\rm }\ &x \in \partial \Omega,
\end{array}\right.
\end{equation*}
Multiply this by $\eta(x)$ and integrate by parts to see that
\begin{eqnarray*}
\int_{\Omega}  \big( \lambda^{*} f'(u^{*}) - \mu_{1}(L^{*},\Omega) \big) \phi \eta ~ \mathrm{d} x =0.
\end{eqnarray*}
Thus there is some $x \in \Omega$ such that
\begin{eqnarray*}
\lambda^{*} f' \big( u^{*}(x) \big) \geqslant \mu_{1}(L^{*},\Omega).
\end{eqnarray*}
Combining this with inequality $(\ref{eq8})$ gives the desired result.
\end{proof}
Combining Theorem $\ref{t4}$ and the obtained lower bound in $(\ref{eq8})$ we conclude that
\begin{eqnarray}\label{eq9}
\dfrac{1}{\psi_{L,\Omega}} \sup_{0 < t < a_{f}}\dfrac{t}{f(t)} \leqslant \lambda^{*}(L,\Omega,f) \leqslant \dfrac{F \big( a_{f} \big)}{\psi_{L,\Omega}}.
\end{eqnarray}
Theorem $\ref{t5}$ illustrates the remarkable usefulness of $(\ref{eq9})$.
\begin{proof}[Proof of Theorem $ \ref{t5} $]
The proof of this theorem is exactly similar to the proof of Theorem 3.1 in \cite{6}. For the convenience of the reader we mention a brief description of the proof.

Take $f_{p}(t):=f(t^{p})$ for $p \geqslant 1$. It is easy to see that there exists a unique $t_{p} > 0$ such that
\begin{eqnarray}\label{eq26}
\frac{t_{p}}{f_{p}(t_{p})} = \displaystyle{\sup_{t > 0}} \frac{t}{f_{p}(t)} \hskip5mm \textrm{for} \hskip1mm \textrm{all} \hskip1mm p \geqslant 1.
\end{eqnarray}
Then, we can show that $t_{p} \longrightarrow 1$ and $t_{p}^{p} \longrightarrow 0$ as $p \rightarrow +\infty$. Therefore
\begin{eqnarray}\label{eq30}
\displaystyle{\lim_{p \rightarrow \infty}\sup_{t > 0}} \frac{t}{f_{p}(t)} = \displaystyle{\lim_{p \rightarrow \infty}} \frac{t_{p}}{f(t_{p}^{p})} = \frac{1}{f(0)}.
\end{eqnarray}
On the other hand
\begin{equation*}
\displaystyle{\lim_{p \rightarrow \infty}} \frac{1}{f_{p}(t)} = \left \{
\begin{array}{ll}
\displaystyle{\frac{1}{f(0)}} & \text{if} \hskip2mm 0 \leqslant t < 1, \\
0 & \text{if} \hskip2mm t > 1.
\end{array} \right.
\end{equation*}
Taking $\zeta: \Bbb{R}_{+} \rightarrow \Bbb{R}_{+}$ with $\zeta(t)=1 / f(0)$ for $t \in [0,1]$ and $\zeta(t)=1 / f_{2}(t)=1 / f(t^{2})$ for $t \in (1,+\infty)$, then $\zeta \in L^{1}(\Bbb{R}_{+})$ and $1 / f_{p}(t) \leqslant \zeta(t)$ for $p \geqslant 2$. Now, by the Lebesgue dominated convergence theorem,
\begin{eqnarray}\label{eq29}
\displaystyle{\lim_{p \rightarrow \infty}} \int_{0}^{\infty} \frac{ds}{f_{p}(s)} = \frac{1}{f(0)}.
\end{eqnarray}
Now, estimate $(\ref{eq9})$ guarantees that
\begin{eqnarray}\label{eq36}
\frac{1}{\psi_{L,\Omega}}\hskip1mm \displaystyle{\sup_{ t > 0 }\frac{t}{f_{p}(t)}} \leqslant \lambda^{*}_{p} \leqslant \frac{1}{\psi_{L,\Omega}} \int_{0}^{u^{*}_{p}(x_{0})} \frac{dt}{f_{p}(t)} \leqslant \frac{1}{\psi_{L,\Omega}} \int_{0}^{+\infty} \frac{dt}{f_{p}(t)}.
\end{eqnarray}
Taking the limit as $p$ tends to infinity in $(\ref{eq36})$ and using $(\ref{eq30})$ and $(\ref{eq29})$, it follows that
\begin{eqnarray*}
\lim_{p \rightarrow \infty} \lambda^{*}_{p} = \frac{1}{f(0) \psi_{L,\Omega}} \hskip5mm \textrm{and} \hskip5mm \lim_{p \rightarrow \infty} u^{*}_{p}(x_{0}) = + \infty,
\end{eqnarray*}
as claimed. 
\end{proof}
In Theorem $\ref{t1}$, by the super-solution method (Lemma $\ref{l3}$) we give a lower bound for the extremal parameter of problem $(\ref{eq1})$.
\begin{proof}[Proof of Theorem $ \ref{t1} $]
Take an $\alpha\in (0,~\frac{||F||_{\infty}}{\psi_\Omega} )$ and define
$\overline{u}(x)=F^{-1}(\alpha \psi_{L}(x))$ for $x\in \Omega.$
It is evident that $\overline{u} \in C^{2}(\Omega) \cap C^{1}(\partial \Omega)$. We show that $\bar{u}$ is a super-solution of $ (\ref{eq1}) $ for $\lambda=\alpha-\alpha^{2}\beta(\alpha)$. To do this, we compute $\Delta\bar{u}(x).$ Note that if we take $y=F^{-1}(\alpha t)$, then it is easy to see that  $y'=\alpha f(y)$ and $y''=\alpha^{2}f(y)f'(y)$. So
\begin{align*}
\Delta\bar{u}(x)&= \Big[ \alpha^{2} f'( \bar{u} ) \big| \nabla \psi_{L}(x) \big|^{2}-\alpha \Big] f( \bar{u} ) \\
&\leqslant \Bigg( \alpha^{2}\sup_{x\in \Omega}f' \Big( F^{-1} \big( \alpha \psi_{L}(x) \big) \Big) \Big| \nabla \psi_{L}(x) \Big|^{2}-\alpha \Bigg) f ( \bar{u} ) \\
&=- \Big( \alpha-\alpha^{2}\beta(\alpha) \Big) f ( \bar{u} ).
\end{align*}
In other words,
$ \Delta \bar{u} (x) + \Big( \alpha-\alpha^{2}\beta(\alpha) \Big) f ( \bar{u} ) \leqslant 0 $, and since we have $\bar{u}(x)=0,~x\in \partial \Omega$, this shows that $\bar{u}$ is a super-solution of $(\ref{eq1})$ for $\lambda=\alpha-\alpha^{2}\beta(\alpha)$, thus, by Lemma $\ref{l3}$, problem $(\ref{eq1})$ with $\lambda=\alpha-\alpha^{2}\beta(\alpha)$ has a classical solution and hence
\begin{eqnarray*}
\lambda^{*}(L,\Omega,f) \geqslant \alpha-\alpha^{2}\beta(\alpha).
\end{eqnarray*}
Taking the supremum over $\alpha\in (0,~\frac{||F||_{\infty}}{\psi_\Omega} )$, we obtain $(\ref{eq10})$.
\end{proof}
Combining Theorem $\ref{t1}$, Theorem $\ref{t4}$ and the  estimates in $ (\ref{eq8}) $, we have
\begin{eqnarray*}
\max \bigg\{ \sup_{0 < \alpha < \frac{||F||_{\infty}}{\psi_{L,\Omega}}} \alpha-\alpha^{2}\beta(\alpha), \dfrac{1}{\psi_{L,\Omega}} \sup_{0 < t < a_{f}}\dfrac{t}{f(t)} \bigg\} \leqslant \lambda^{*}(L,\Omega,f),
\end{eqnarray*}
and
\begin{eqnarray*}
\lambda^{*}(L,\Omega,f) \leqslant \min \bigg\{ \dfrac{F \big( a_{f} \big)}{\psi_{L,\Omega}}, \mu_{1}(L^{*},\Omega) \sup_{0 < t < a_{f}} \dfrac{t}{f(t)} \bigg\}.
\end{eqnarray*}
where $\beta(\alpha):=\displaystyle{\sup_{x \in \Omega}}f' \Big( F^{-1} \big( \alpha \psi_{L}(x) \big) \Big) \big| \nabla \psi_{L}(x) \big|^{2} $.

In the following two examples, we apply the above results for standard nonlinearities $f(u)=e^{u}$ (as a regular nonlinearity) and $ f(u) = \dfrac{1}{(1-u)^{2}} $ (as a singular nonlinearity) on the unit ball $ B(0,1) \subseteq \Bbb{R}^{N} $.
\begin{example}\label{e1}
Consider the following problem
\begin{equation}\label{eq27}
\left\{\begin{array}{ll} - \triangle u + \dfrac{-2x \centerdot \nabla u}{1+|x|^{2}} = \lambda e^{u}  {\rm }\ &x \in B(0,1), \\
\hskip25mm u = 0 {\rm }\ &x \in \partial B(0,1).
\end{array}\right.
\end{equation}
Here, we have
\begin{eqnarray*}
c(x)=\dfrac{-2x}{1+|x|^{2}},  ~~~~~~~ L=-\triangle + c(x) \centerdot \nabla,  ~~~~~~~ f(u)=e^{u},  ~~~~~~~ \Omega = B(0,1).
\end{eqnarray*}
Now, we look for radial solution for torsion function $\psi_{L}$. If there exists smooth function $\varphi:[0,1] \rightarrow \Bbb{R}$ such that $\psi_{L} = \varphi (|x|)$, then it is easy to see that $\varphi$ satisfies the following
\begin{equation*}
\left\{\begin{array}{ll} \varphi ''(|x|) + \Big( \dfrac{N-1}{|x|} + \dfrac{2 |x|}{1+|x|^{2}} \Big) \varphi ' (|x|) = -1 \\
\hskip49.5mm \varphi (1) = 0.
\end{array}\right.
\end{equation*}
Solving the above problem, we get
\begin{eqnarray*}
\psi_{L}(x) = \varphi (|x|) = \dfrac{N (1- |x|^{2}) + 2 \ln \Big( \dfrac{2}{1+|x|^{2}} \Big)}{2 N (N+2)}.
\end{eqnarray*}
Thus
\begin{eqnarray*}
\psi_{L,\Omega} = \psi_{L}(0) = \dfrac{N+ \ln 4}{2 N (N+2)}.
\end{eqnarray*}
By $(\ref{eq9})$, we have
\begin{eqnarray}\label{eq28}
\dfrac{2 N (N+2)}{e (N + \ln 4)} \leqslant \lambda^{*} (L, B(0,1), e^{u}) \leqslant \dfrac{2 N (N+2)}{N + \ln 4}.
\end{eqnarray}
One can also apply Theorem $\ref{t1}$ to obtain another lower bound for the extremal parameter of problem $(\ref{eq27})$. Here, we have
\begin{eqnarray*}
f'(t)=e^{t}, ~~~~~~~ F(t)=1-e^{-t}, ~~~~~~~ F^{-1}(t)=-\ln (1-t).
\end{eqnarray*}
Thus
\begin{eqnarray*}
\beta (\alpha) = \sup_{x \in B(0,1)} \dfrac{\big| \nabla \psi_{L}(x) \big|^{2}}{1-\alpha \psi_{L}(x)} = \sup_{0 < t < 1}\dfrac{\varphi '^{2}(t)}{1-\alpha \varphi (t)} ~~~ \textrm{for} ~ \textrm{all} ~ 0 < \alpha < \dfrac{2 N (N+2)}{N + \ln 4}.
\end{eqnarray*}
It can be easily checked that
\begin{eqnarray*}
\beta (\alpha) = \varphi '^{2}(1) = \dfrac{(N+1)^{2}}{N^{2} (N+2)^{2}} ~~~ \textrm{for} ~ \textrm{all} ~ 0 < \alpha < \dfrac{2 N^{2} (N+2)}{(N+1)^{2}}.
\end{eqnarray*}
On the other hand
\begin{eqnarray*}
\sup_{0 < \alpha < \dfrac{2 N (N+2)}{N + \ln 4}} \alpha - \alpha ^{2} \beta (\alpha) \geqslant \sup_{0 < \alpha < \dfrac{2 N^{2} (N+2)}{(N+1)^{2}}} \alpha - \alpha^{2} \beta(\alpha) = \dfrac{2 N^{3}}{(N+1)^{2}}.
\end{eqnarray*}
Hence
\begin{eqnarray*}
\dfrac{2 N^{3}}{(N+1)^{2}} \leqslant \lambda^{*}(L,B(0,1),e^{u}).
\end{eqnarray*}
Note that this lower bound is better than the one in $(\ref{eq28})$ for all $N \geqslant 3$.
\end{example}
\begin{example}
Consider the following problem
\begin{equation}\label{eq35}
\left\{\begin{array}{ll} - \triangle u + \dfrac{-2x \centerdot \nabla u}{1+|x|^{2}} = \dfrac{\lambda}{(1-u)^{2}}  {\rm }\ &x \in B(0,1), \\
\hskip25mm u = 0 {\rm }\ &x \in \partial B(0,1).
\end{array}\right.
\end{equation}
By Example $ \ref{e1} $, we know that
\begin{eqnarray*}
\psi_{L} = \varphi (|x|) = \dfrac{N (1- |x|^{2}) + 2 \ln \Big( \dfrac{2}{1+|x|^{2}} \Big)}{2 N (N+2)} ~~~ \textrm{and} ~~~ \psi_{L,\Omega} = \psi_{L}(0) = \dfrac{N+ \ln 4}{2 N (N+2)}.
\end{eqnarray*}
By $(\ref{eq9})$, we have
\begin{eqnarray}
\dfrac{8 N (N+2)}{27 (N + \ln 4)} \leqslant \lambda^{*} \big( L, B(0,1), (1-u)^{-2} \big) \leqslant \dfrac{2 N (N+2)}{3(N + \ln 4)}.
\end{eqnarray}
Again, one can also apply Theorem $\ref{t1}$ to obtain another lower bound for the extremal parameter of problem $(\ref{eq35})$.
It can be easily checked that
\begin{eqnarray*}
\beta (\alpha) = 2 \varphi '^{2}(1) = \dfrac{2 (N+1)^{2}}{N^{2} (N+2)^{2}} ~~~ \textrm{for} ~ \textrm{all} ~ 0 < \alpha < \dfrac{2 N^{2} (N+2)}{3 (N+1)^{2}}.
\end{eqnarray*}
On the other hand
\begin{eqnarray*}
\sup_{0 < \alpha < \dfrac{2 N (N+2)}{3 (N + \ln 4)}} \alpha - \alpha ^{2} \beta (\alpha) \geqslant \sup_{0 < \alpha < \dfrac{2 N^{2} (N+2)}{3 (N+1)^{2}}} \alpha - \alpha^{2} \beta(\alpha) = \dfrac{2 N^{2}(3 N + 2)}{9 (N+1)^{2}}.
\end{eqnarray*}
Hence
\begin{eqnarray*}
\dfrac{2 N^{2}(3N+2)}{9(N+1)^{2}} \leqslant \lambda^{*}(L,B(0,1),(1-u)^{-2}).
\end{eqnarray*}
Note that this lower bound is better than the one in $(\ref{eq28})$ for all $N \geqslant 2$.
\end{example}
We conclude this section, by proving the last assertion of Proposition $\ref{p2}$.
\begin{proof}[Proof of Proposition $ \ref{p2} $ $ (ii) $]
Set $\lambda(\alpha) = \alpha - \alpha^{2} \beta (\alpha)$ for all $0 < \alpha < \dfrac{\|F\|_{\infty}}{\psi_{L,\Omega}}$, where
\begin{eqnarray*}
\beta(\alpha):=\displaystyle{\sup_{x \in \Omega}}f' \Big( F^{-1} \big( \alpha \psi_{L}(x) \big) \Big) \big| \nabla \psi_{L}(x) \big|^{2}. 
\end{eqnarray*}
Clearly $ \lambda (\alpha) \longrightarrow 0 $ as $ \alpha \longrightarrow 0^{+} $. By Theorem $ \ref{t4} $ and Theorem $ \ref{t1} $ we have
\begin{eqnarray}\label{eq11}
\psi_{L} (x) \leqslant \dfrac{u_{ \lambda ( \alpha ) }(x)}{ \lambda ( \alpha ) } \leqslant \dfrac{1}{1- \alpha  \beta ( \alpha ) } \psi_{L} (x) ~~~ \textrm{for} ~ \textrm{all} ~ x \in \overline{\Omega}.
\end{eqnarray}
Taking the limit on both sides of $(\ref{eq11})$ as $\alpha \longrightarrow 0^{+}$, we then have the conclusion of Proposition  $\ref{p2}$ $ (ii) $.
\end{proof}
%
\section{Application to the explosion problem in a flow}
%
In this section, we apply previous results to the explosion problem in a flow. First, we determine the behavior of the extremal parameter of problem $ (\ref{eq12}) $ when the flow $ c(x) $ is divergence-free (see Theorem $ \ref{t3} $) and then we prove Theorem $ \ref{t6} $.
\begin{proof}[Proof of Theorem $ \ref{t3} $]
By Lemma $\ref{l2}$ and Theorem $\ref{t4}$ we have
\begin{eqnarray}\label{eq13}
\dfrac{1}{\psi_{A,\Omega}} \sup_{t > 0} \dfrac{t}{f(t)} \leqslant \lambda^{*}(A) \leqslant \dfrac{F \big( a_{f} \big)}{\psi_{A,\Omega}}.
\end{eqnarray}
Now, by Theorem A and estimate $(\ref{eq13})$ the proof of Theorem $\ref{t3}$ is complete.
\end{proof}
Theorem $\ref{t3}$ completely determine the behaviour of extremal parameter of problem $(\ref{eq12})$ when $c(x)$ is divergence-free. But there is still another interesting case when $c(x)$ is not divergence-free. As it is mentioned, in Theorem $\ref{t6}$, we completely determine the behaviour of extremal parameter of problem $(\ref{eq12})$ for a wide class of flows $c(x)$ which are not divergence-free.
\begin{proof}[Proof of Theorem $ \ref{t6} $]
Define
\begin{eqnarray*}
\varphi (r) = \int_{0}^{1} \dfrac{\int_{0}^{t}s^{N-1}g^{A}(s) \mathrm{d} s}{t^{N-1}g^{A}(t)} \mathrm{d} t - \int_{0}^{r} \dfrac{\int_{0}^{t}s^{N-1}g^{A}(s) \mathrm{d} s}{t^{N-1}g^{A}(t)} \mathrm{d} t ~~~ \textrm{for} ~ \textrm{all} ~ 0 \leqslant r \leqslant 1,
\end{eqnarray*}
where
\begin{eqnarray*}
g(r):= e^{\displaystyle{\int_{0}^{r}s  \rho(s) ds}}~~~\textrm{for} ~ 0 \leqslant r \leqslant 1.
\end{eqnarray*}
Then (as it is described in Example $\ref{e1}$ ) it is not hard to check that $\psi_{L_{A}}=\varphi (|x|)$ and since the function
\begin{eqnarray*}
r \longmapsto \int_{0}^{r} \dfrac{\int_{0}^{t}s^{N-1}g^{A}(s) \mathrm{d} s}{t^{N-1}g^{A}(t)} \mathrm{d} t ~~~ \textrm{for} ~ \textrm{all} ~ 0 \leqslant r \leqslant 1,
\end{eqnarray*}
is increasing, so
\begin{eqnarray}\label{eq15}
\psi_{L_{A},B} = \int_{0}^{1} \dfrac{\int_{0}^{t}s^{N-1}g^{A}(s) \mathrm{d} s}{t^{N-1}g^{A}(t)} \mathrm{d} t.
\end{eqnarray}
Making the change of variable $s=t h$ in the interior integral in $(\ref{eq15})$, we get
\begin{eqnarray*}
\int_{0}^{t}s^{N-1}g^{A}(s) \mathrm{d} s = t^{N} \int_{0}^{1} h^{N-1}g^{A}(t h) \mathrm{d} h.
\end{eqnarray*}
Thus
\begin{eqnarray*}
\psi_{L_{A},B} =  \int_{0}^{1} t~ \dfrac{\int_{0}^{1}h^{N-1}g^{A}(t h) \mathrm{d} h}{g^{A}(t)} \mathrm{d} t.
\end{eqnarray*}
\begin{itemize}
\item[(i)] If there esits $x_{0} \in [0,1]$ such that $\rho(x_{0}) < 0$, then the continuity of $\rho$ implies that there exits an interval $I = [a,b] \subseteq [0,1]$ such that $\rho$ is negative on $I$. This means that the function $g$ defined above is strictly decreasing on $I$. Choose an $\epsilon>0$ such that
\begin{eqnarray}\label{eq16}
0 < \frac{3}{2} - \sqrt{\dfrac{2a}{b} + \dfrac{1}{4}} < \epsilon < 1 - \dfrac{a}{b} < 1.
\end{eqnarray}
It is easy to see that inequality $(\ref{eq16})$ implies that
\begin{eqnarray}\label{eq19}
(1-\dfrac{\epsilon}{2})(1-\epsilon) < \dfrac{a}{b} < 1-\epsilon.
\end{eqnarray}
Now, since the function $g$ is strictly decreasing on $I$, then for all $A \geqslant 1$ we have
\begin{align}\label{eq21}
\psi_{L_{A},B} = \int_{0}^{1} t \dfrac{\int_{0}^{1}h^{N-1}g^{A}(t h) \mathrm{d} h}{g^{A}(t)} \mathrm{d} t &\geqslant  \int_{a / (1-\epsilon)}^{b}  \dfrac{\displaystyle{t \int_{1-\epsilon}^{1-(\epsilon / 2)}}h^{N-1}g^{A}(t h) \mathrm{d} h}{g^{A}(t)} \mathrm{d} t \nonumber \\ &\geqslant \int_{a / (1-\epsilon)}^{b} t \Bigg( \dfrac{g \big( t(1-(\epsilon / 2)) \big)}{g(t)} \Bigg)^{A} \mathrm{d} t ~ \int_{1-\epsilon}^{1-(\epsilon / 2)} h^{N-1} \mathrm{d} h \nonumber \\ &\geqslant \Bigg( \dfrac{g \big( b(1-(\epsilon / 2)) \big)}{g(a / (1-\epsilon))} \Bigg)^{A} \int_{a / (1-\epsilon)}^{b} t \mathrm{d} t \int_{1-\epsilon}^{1-(\epsilon / 2)} h^{N-1} dh.
\end{align}
By $(\ref{eq19})$, we know that $a < b(1- (\epsilon / 2)) < a / (1-\epsilon) < b$, therefore
\begin{eqnarray*}
\dfrac{g \big( b(1-(\epsilon / 2)) \big)}{g(a / (1-\epsilon))} > 1.
\end{eqnarray*}
Thus
\begin{eqnarray*}
 \Bigg( \dfrac{g \big( b(1-(\epsilon / 2)) \big)}{g(a / (1-\epsilon))} \Bigg)^{A} \longrightarrow \infty ~~ \textrm{as} ~ A \longrightarrow \infty.
\end{eqnarray*}
Now $(\ref{eq21})$ guarantees that $\psi_{L_{A},B} \longrightarrow \infty$ as $A \longrightarrow \infty$.
\item[(ii)] If $\rho \geqslant 0$ and $\rho \not\equiv 0$ on any interval $I \subseteq [0,1]$, then $g$ is strictly increasing.
Let $0 < \epsilon < 1$ be arbitrary, then
\begin{align*}
\int_{0}^{1} h^{N-1} g^{A}(t h) \mathrm{d} h &= \int_{0}^{\epsilon} h^{N-1} g^{A}(h t) \mathrm{d} h + \int_{\epsilon}^{1} h^{N-1} g^{A}(h t) \mathrm{d} h \\
&\leqslant \epsilon^{N} g^{A}(\epsilon t) + (1-\epsilon) g^{A}(t).
\end{align*}
It then  follows  that
\begin{eqnarray*}
\psi_{L_{A},B} \leqslant \dfrac{1-\epsilon}{2} + \epsilon^{N} \int_{0}^{1} t \Big( \dfrac{g(\epsilon t)}{g(t)} \Big)^{A} \mathrm{d} t.
\end{eqnarray*}
Since $g$ is strictly increasing, it is evident that $t \Big( \dfrac{g(\epsilon t)}{g(t)} \Big)^{A} \longrightarrow 0$ as $A \longrightarrow \infty$ pointwise for all $0 \leqslant t \leqslant 1$, on the other hand $t \Big( \dfrac{g(\epsilon t)}{g(t)\epsilon} \Big)^{A} \leqslant t \in L^{1}([0,1])$ for all $A \geqslant 0$. Now, Lebesgue dominated convergence theorem implies that
\begin{eqnarray*}
\lim_{A \rightarrow \infty} \epsilon^{N} \int_{0}^{1} t \Big( \dfrac{g(\epsilon t)}{g(t)} \Big)^{A} \mathrm{d} t = 0.
\end{eqnarray*}
Thus
\begin{eqnarray*}
\limsup_{A \rightarrow \infty} \psi_{L_{A},B} \leqslant \dfrac{1-\epsilon}{2}.
\end{eqnarray*}
Letting $\epsilon \longrightarrow 1^{-}$ in the above inequality, we get $\psi_{L_{A},B} \longrightarrow 0$ as $A \longrightarrow \infty$.
\item[(iii)] If $\rho \geqslant 0$ and $\rho \equiv 0$ on some interval $[a,b] \subseteq [0,1]$, then $g$ is constant on $[a,b]$. Since the function $g$ is increasing on $[0,1]$, then
\begin{eqnarray}\label{eq22}
\psi_{L_{A},B} = \int_{0}^{1} t \dfrac{\int_{0}^{1}h^{N-1}g^{A}(t h) \mathrm{d} h}{g^{A}(t)} \mathrm{d} t \leqslant \int_{0}^{1}t\dfrac{g^{A}(t) \int_{0}^{1}h^{N-1} \mathrm{d} h}{g^{A}(t)} \mathrm{d} t = \dfrac{1}{2 N} ~~~ \textrm{for} ~ \textrm{all} ~ A \geqslant 0.
\end{eqnarray}
On the other hand, since $g$ is constant on $[a,b]$ we have
\begin{align}\label{eq23}
\psi_{L_{A},B} = \int_{0}^{1} \dfrac{\int_{0}^{t}s^{N-1}g^{A}(s) \mathrm{d} s}{t^{N-1}g^{A}(t)} \mathrm{d} t &\geqslant \int_{a}^{b} \dfrac{\int_{a}^{t} s^{N-1} g^{A}(s) \mathrm{d} s}{t^{N-1} g^{A}(t)} \mathrm{d} t \nonumber \\ &=\int_{a}^{b} \dfrac{\int_{a}^{t} s^{N-1} \mathrm{d} s}{t^{N-1}} \mathrm{d} t \nonumber \\ &= \dfrac{1}{N} \int_{a}^{b}\dfrac{t^{N}-a^{N}}{t^{N-1}} \mathrm{d} t ~~~ \textrm{for} ~ \textrm{all} ~ A \geqslant 0.
\end{align}
By $(\ref{eq22})$ and $(\ref{eq23})$ we conclude that
\begin{eqnarray*}
C_{N,\rho}:=\dfrac{1}{N} \int_{a}^{b}\dfrac{t^{N}-a^{N}}{t^{N-1}} \mathrm{d} t \leqslant \psi_{L_{A},B} \leqslant \dfrac{1}{2N} ~~~ \textrm{for} ~ \textrm{all} ~ A \geqslant 0,
\end{eqnarray*}
that completes the proof.
\end{itemize}
\end{proof}
%

\end{document}